\documentclass{amsart}

\usepackage{lineno}
\usepackage{amssymb}
\usepackage{amsmath}
\usepackage{amsthm}
\usepackage{mathrsfs}
\usepackage{bbm}
\usepackage{cite}
\usepackage{color}
\usepackage{enumitem}
\usepackage{pgfplots}
\usepackage[utf8]{inputenc}

\usepackage[colorlinks,hypertexnames=false]{hyperref}
\usepackage[msc-links]{amsrefs}

\usepackage{changes}

\usepackage{mathtools}
\usepackage{extarrows}
\usepackage{setspace}

\usepackage{graphicx}

\usepackage{imakeidx}
\makeindex

\hypersetup{
	colorlinks=true,
	linkcolor=black,
	filecolor=black,
	urlcolor=black,
	citecolor=black,
}

\newtheorem{thm}{Theorem}[section]

\newtheorem{defn}[thm]{Definition}

\newtheorem{theorem}[thm]{Theorem}
\newtheorem{corollary}[thm]{Corollary}

\newtheorem{remark}[thm]{Remark}
\newtheorem{example}[thm]{Example}

\numberwithin{equation}{section}

\pgfplotsset{compat=1.18} 

\begin{document}

\title[]{On the Eneström-Kakeya Theorem for polynomials of an octonionic variable}
\author{Ting Yang}
\email[Ting Yang]{tingy@ustc.mail.edu.cn}
\address{School of Mathematics and Statistics, Anqing Normal University, Anqing, 246133, China}
\author{Xinyuan Dou}
\email[Xinyuan Dou]{douxinyuan@ustc.edu.cn}
\address{Department of Mathematics, University of Science and Technology of China, Hefei 230026, China}

%\date{\today}


\begin{abstract}
	To study the zeros of octonionic polynomials, we generalize the well-known Enestr\"{o}m-Kakeya Theorem to the case of octonions. In this paper, we first deal with octonionic polynomials with nonnegative and monotonic coefficients, and prove that its zero set is contained within the closed sphere of octonion space. Then, we also consider the octonionic polynomials which the coefficients muduli is monotonic and the real parts of the coefficients is monotonic respectively, and get some results.
\end{abstract}

\maketitle
\section{Introduction}
The Enestr\"{o}m-Kakeya Theorem states the distribution of the zeros of polynomials with complex variable as follows:
  \begin{theorem}\label{E-K-O}
   Let $p(z)=\sum_{k=0}^na_kz^k$ be a polynomial of degree $n$  with real coefficients satisfying 
    $$0< a_0\leq a_1\leq \dots\leq a_n,$$
    then all the zeros of $p$ lie in $|z|\leq 1.$
\end{theorem}
This work was first published in 1893 by G. Enestr\"{o}m
\cite {enestrom1893harledning}.
In the following three years, he mentioned his result again in publications. In 1912, S. Kakeya removed the restriction on the monotonicity of coefficients, and obtained a more general result \cite{kakeya1912limits}. Later, the generalization of  the Eneström-Kakeya Theorem has yielded abundant results.  For instance, A. Joyal, G. Labelle, and Q. I. Rahman
dropped the condition of nonnegativity and maintained the condition of monotonicity \cite{joyal1967location}.
N. K. Govil and Q. I. Rahman presented a
result that is applicable to polynomials with complex coefficients \cite{govil1968enestrom}.
\begin{theorem}
     Let $p(z)=\sum_{k=o}^na_kz^k\not\equiv 0$ be a polynomial of degree $n$ with complex coefficients,  such that
    \begin{equation*}
      |\arg a_k\ \beta|\leq \alpha\leq \pi/2,\quad  k=0,1,2,\dots n.
    \end{equation*}
    for some real $\beta,$ and
    $$|a_n|\geq |a_{n-1}| \geq |a_{n-2}|  \geq\dots  \geq|a_0| .$$
    Then $p(z)$ has all its zeros on or inside the circle
    $$|z|=\cos{\alpha}+\sin{\alpha}+\frac{2\sin{\alpha}}{|a_n|}\sum_{k=0}^{n-1}|a_k|.$$
\end{theorem}
The above theorem reduces to  Theorem \ref{E-K-O} when  $\alpha=\beta=0$. In this paper,
they also prove the following.
\begin{theorem}
    Let $p(z)=\sum_{k=0}^na_kz^k\not\equiv 0$ be a polynomial of degree $n$. If $\Re a_k=\alpha _k,\ \Im a_k=\beta_k\  \text{for}\ k=0,1,2,\dots,n.$ and 
\begin{equation*}
   \alpha _n \geq \alpha _{n-1}\geq\dots  \geq \alpha _1   \geq\alpha _0\geq 0,\  \alpha _n>0,
\end{equation*}
    then $p(z)$ does not vanish in  
    $$|z|> 1+\frac{2}{\alpha_n}\sum_{k=0}^n|\beta_k|.$$
\end{theorem}
 This result reduces to Theorem \ref{E-K-O} when $\beta_k=0\  \text{for}\ k=0,1,2,\dots,n$ . The development process of Enestr\"{o}m-Kakeya Theorem is detailed in the literature \cite{2014Current}.
Later,  the Enestr\"{o}m-Kakeya Theorem has been widely generalized from the complex to quaternionic setting (\cites{tripathi2020note,carney2020enestrom,gardner2022generalization,milovanovic2022zeros,hussain2023enestrom,mir2024estimation,gardner2025type}, etc.)

In this paper, we are going to generalize the Enestr\"{o}m-Kakeya Theorem to the setting of octonions $\mathbb{O}$, which is a nonassociative, alternative division algebra.

  \section{Priliminaries}

 The algebra of octonions \cite{baez2002octonions} can be constructed by considering a basis $\varepsilon=\{e_0,e_1,\dots,e_6,e_7\}$ of $\mathbb{R}^8$ and relations 
 $$e_{\alpha}e_{\beta}=-\delta_{\alpha \beta }+\psi_{\alpha \beta \gamma }e_{\gamma },\ \alpha,\beta,\gamma =1,2,\dots,7,$$
where $e_0=1$, $\delta _{\alpha \beta}$ is the Kronecker delta, $\psi_{\alpha\beta\gamma}$ is totally antisymmetric in $\alpha,\beta,\gamma,$ and $\psi_{\alpha\beta\gamma}$ equals $1$ if $(\alpha,\beta,\gamma)$ belongs to the set
$$\sigma =\{(1,2,3),(1,4,5),(2,4,6),(3,4,7),(2,5,7),(1,6,7),(5,3,6)\}.$$
Every element in $\mathbb{O}$ can be expressed as $\omega=x_0+\sum_{k=1}^7x_ke_k,$ and its square norm $|\omega|^2=\sum_{k=1}^7x_k^2.$ 

The set $(1,e_1,e_2,e_1e_2)$ form a basis for a subalgebra of $\mathbb{O}$ isomorphic to the algebra $\mathbb{H}$ of quaternions. Moreover, every octonion can viewed as four complex numbers or as two quternions. Indeed,

\begin{equation*}
\begin{split}
\mathbb{O}&=\mathbb{R}+\mathbb{R}e_1+\mathbb{R}e_2+\mathbb{R}e_3+\mathbb{R}e_4+\mathbb{R}e_5+\mathbb{R}e_6+\mathbb{R}e_7\\
&=(\mathbb{R}+\mathbb{R}e_1)+(\mathbb{R}+\mathbb{R}e_1)e_2+[(\mathbb{R}+\mathbb{R}e_1)+(\mathbb{R}+\mathbb{R}e_1)e_2]e_4 \\
&=\mathbb{C}+\mathbb{C}e_2+(\mathbb{C}+\mathbb{C}e_2)e_4\\
&=\mathbb{H}+\mathbb{H}e_4.
\end{split}
\end{equation*}
Let $\mathbb{S}$ denote the unit sphere of purely imaginary octonions, i.e.,
$$\mathbb{S}=\{\omega=\sum_{k=1}^7x_ke_k:\sum_{k=1}^7x_k^2=1\}.$$
It is known that if $I\in \mathbb{S},$ then $I^2=-1.$

The definition of slice regular functions of a octonionic variable as follows (see \cites{gentili2010regular,gentili2013regular}):
\begin{defn}
    Let $\Omega $ be a domian in $\mathbb{O}.$ A real differentiable function  $f:\Omega \rightarrow{\mathbb{O}}$ is said to be $($slice$)$ regular if, for every $I\in \mathbb{S},$ its restriction $f_I$ to the complex line $L_I=\mathbb{R}+\mathbb{R}_I$ passing through the origin and containing $1$ and $I$ is holomorphic in $\Omega_I=\Omega \bigcap L_I;$ namely,
    $$\bar{\partial}_If(x+yI)=\frac{1}{2}\left(\frac{\partial}{\partial x}+I\frac{\partial}{\partial y}\right)f_I(x+yI)$$
    vanishes identically.
\end{defn}
 \begin{example}
   Let $f(q)=a_0+qa_1+q^2a_2+\dots+q^na_n,\ a_l\in \mathbb{O}$ for all $l\in \mathbb{N}$, then $f$ is slice regular on $\mathbb{O}.$
 \end{example}
For all $R\in (0,+\infty],$ denote
$$\mathbb{B}(0,R)=\{q\in \mathbb{O}:\ |q|<R\}.$$

Since the product of two regular functions is not necessarily regular, we need the following multiplication.
\begin{defn}
    Let $f,g :\mathbb{B}(0,R)\rightarrow{\mathbb{O}}$ be regular functions and let $f(q)=\sum_{n\in \mathbb{N}}q^na_n,\ g(q)=\sum_{n\in \mathbb{N}}q^nb_n$ be their power series expansions. The regular product of $f$ and $g$ is the regular functions defined by
\begin{equation}\label{star product}
        f*g(q)=\sum_{n\in \mathbb{N}}q^n\left(\sum_{k=0}^n a_kb_{n-k}\right)
\end{equation}
    on $\mathbb{B}(0,R).$
\end{defn}

 In the study of extending Enestr\"{o}m-Kakeya Theorem to octonion algebra, we also need the corresponding Maximum Modulus Principle.
\begin{theorem}$($Maximum Modulus Principle$)$ \label{MMP}
Let $f:\mathbb{B}(0,R)\rightarrow{\mathbb{O}}$ be regular function. If $|f|$ has a relative maximum at a point $a\in \mathbb{B}(0,R),$ then $f$ is constant on $\mathbb{B}(0,R).$
\end{theorem}

\section{Main results}
Now we can generalize Theorem \ref{E-K-O} to polynomials of a octonionic
variable.
\begin{theorem}\label{E-K}
   Let $p(q)=a_0+qa_1+q^2a_2+\dots+q^na_n$ be a polynomial of degree $n$ with octonionic variable and real coefficients satisfying 
    $$0< a_0\leq a_1\leq \dots\leq a_n.$$
    Then all the zeros of $p$ lie in $|q|\leq 1.$
\end{theorem}
\begin{proof}
    Let \begin{equation*}
        p(q)=a_0+qa_1+q^2a_2+\dots+q^na_n\quad \text{with}\ q\in \mathbb{O},
    \end{equation*}
   where $a_k\in \mathbb{R}$ for $k=0,1,\dots n.$ Then $p(q)$ is a regular function on $\mathbb{O}.$
     By the regular product of regular functions, we get
\begin{equation}\label{real coe}
    p(q)*(1-q)=a_0+qa_1+q^2a_2+\dots+q^na_n-q(a_0+qa_1+q^2a_2+\dots+q^na_n).
\end{equation}

Due to the non-associativity of octonions, using to the  Artin Theorem in \cite{schafer2017introduction}, we get
$$q(q^ka_k)=q^{k+1}a_k\quad \text{for any}\ q\in \mathbb{O},a_k\in \mathbb{R} \ \text{and}\ k=0,1,\dots,n. $$
It is obvious that 
$$p(q)*(1-q)=a_0+q(a_1-a_0)+q^2(a_2-a_1)+\dots+q^n(a_n-a_{n-1})-q^{n+1}a_n.$$
     
Denote
    $$f(q):=a_0+q(a_1-a_0)+q^2(a_2-a_1)+\dots+q^n(a_n-a_{n-1}), $$ 
    then $f(q)=p(q)*(1-q)+q^{n+1}a_n,$ which is a regular function on $\mathbb{O}.$
From equation $(\ref{real coe}),$ 
$$p(q)*(1-q)=p(q)-qp(q)=(1-q)p(q).$$
Hence
\begin{equation}\label{zero}
    p(q)*(1-q)=0  \ \text{if and only if }\ p(q)=0\ \text{or }q=1.
\end{equation}

As a similar method in the case of complex variable function \cite{2014Current},
for $|q|=1,$ we have
\begin{equation*}
    \begin{split}
       |f(q)|&=|a_0+\sum_{k=1}^nq^k(a_k-a_{k-1})|\\
       &\leq  |a_0|+\sum_{k=1}^n|q^k(a_k-a_{k-1})|\\
       &=|a_0|+\sum_{k=1}^n|(a_k-a_{k-1})|
    \end{split}.
\end{equation*}
Due to the non-negative monotonicity of coefficients of this polynomial,
$$|f(q)|\leq a_0+\sum_{k=1}^n(a_k-a_{k-1})=a_n\quad \text{for all}\ |q|=1.$$
Let $g(q)=q^nf\left (1/q\right)$ for all $q\in \mathbb{O}.$ Then
$g(q)$ is also a octonionic polynomial of degree $n$ with  real coefficients, which is a regular function on $|q|\leq 1.$
Note that
$$\max_{|q|=1}|g(q)|=\max_{|q|=1}\left|f\left (1/q\right)\right|=\max_{|q|=1}|f(q)|\leq a_n.$$
By Theorem \ref{MMP}, we have
$|g(q)|\leq a_n$ on $|q|\leq 1.$ That is, 
$$\left|f\left (1/q\right)\right|\leq \frac {a_n}{|q|^n}\quad \text{for all}\ 0<|q|\leq 1.$$ 
Hence,
$$\left|f(q)\right|\leq a_n|q|^n\quad \text{for all}\ |q|\geq 1.$$
Then for all $|q|\geq 1,$
\begin{equation*}
    \begin{split}
       \left| p(q)*(1-q)\right|&=|f(q)-q^{n+1}a_n|\\
       &\geq a_n|q|^{n+1}-|f(q)|\\
       &\geq a_n|q|^{n+1}-a_n|q|^{n}\\
       &=a_n|q^{n}|(|q|-1).
    \end{split}
\end{equation*}
This implies that $\left| p(q)*(1-q)\right|>0$ for each $|q|>1.$  According to equation (\ref{zero}), $|p(q)|>0$ for all $|q|>1.$ Therefore, all the zeros of $p(q)$ lie in $|q|\leq 1.$
\end{proof}
Taking into account $p(q/a)$ for some $a>0$ in Theorem \ref{E-K}, we have a more general result as follows:
\begin{corollary}
  Let $p(q)=a_0+qa_1+q^2a_2+\dots+q^na_n$ be a polynomial of degree $n$ with octonionic variable and real coefficients,  satisfying 
    \begin{equation*}
        0< a_0a^n\leq a_1a^{n-1}\leq \dots\leq a_{n-2}a^2\leq a_{n-1}a\leq a_n
    \end{equation*}
    for some $a>0.$
    Then all the zeros of $p$ lie in $|q|\leq 1/a.$
\end{corollary}

In the theorem above, we can replace the non-negative monotonicity of the coefficients with the monotonicity of the coefficient moduli.

\begin{theorem}\label{a}
    Let $p(q)=a_0+qa_1+q^2a_2+\dots+q^na_n\not \equiv  0$ be a polynomial of degree $n$ with octonionic variable and octonionic coefficients,  satisfying 
    \begin{equation*}
      |a_0|a^n\leq |a_1|a^{n-1}\leq \dots\leq |a_{n-2}|a^2\leq |a_{n-1}|a\leq |a_n|
    \end{equation*}
    for some $a>0.$
    Then all the zeros of $p$ lie in $|q|\leq \frac{1}{a}K_1,$ where $K_1$ is the greatest positive root of the trinomial equation 
    $$K^{n+1}-2K^n+1=0.$$
\end{theorem}
\begin{proof}
   Let $$p(q)=a_0+qa_1+q^2a_2+\dots+q^na_n\quad \text{with}\ q\in \mathbb{O}\ \text{and}\ a_l\in \mathbb{O}.$$
Denote
$$f(q):=a_0+qa_1+q^2a_2+\dots+q^{n-1}a_{n-1},$$
then $p(q)=f(q)+q^na_n.$

Note that 
$$\frac{|a_k|}{|a_{n-1}|}\leq \frac{1}{a^{n-k-1}}\quad \text{for}\ k=0,1,2,\dots,n-2.$$ 
Then for $|q|=R\ \left (>\frac{1}{a}\right),$ we have
\begin{equation*}
    \begin{split}
      |f(q)|&\leq |a_0|+|qa_1|+|q^2a_2|+\dots +|q^{n-2}a_{n-2}|+|q^{n-1}a_{n-1}| \\&=  R^{n-1}\left|a_{n-1}\right |\cdot \left (1+\frac{1}{R}\frac{|a_{n-2}|}{|a_{n-1}|}+\dots +\frac{1}{R^{n-2}}\frac{|a_1|}{|a_{n-1}|}+\frac{1}{R^{n-1}}\frac{|a_0|}{|a_{n-1}|}\right) \\
      &\leq R^{n-1}\left|a_{n-1}\right |\cdot \left (1+\frac{1}{aR} +\dots +\frac{1}{(aR))^{n-2}} +\frac{1}{(aR)^{n-1}}\right) \\
      &=|a_{n-1}| \frac{(aR)^{n}-1}{a^{n-1}(aR-1)}.
    \end{split}
\end{equation*}
Hence, if $|q|=R\ \left (>\frac{1}{a}\right),$ then 
\begin{equation*}
    \begin{split}
        |p(q)|&= |f(q)+q^na_n|\\
        &\geq |q^na_n|-|f(q)|\\
        &\geq  |a_{n}|R^{n}-\left |a_{n-1}\right | \frac{(aR)^{n}-1}{a^{n-1}(aR-1)}\\
        &>0
    \end{split}
\end{equation*}
   if $$\frac{|a_{n}|}{a|a_{n-1}|}\geq \frac{(aR)^{n}-1}{(aR)^n(aR-1)}.$$
Since $$\frac{|a_{n}|}{a|a_{n-1}|}\geq 1,$$
 implies that if $(aR)^{n}-1<(aR)^n(aR-1),$ we have $|p(q)|>0 $ on $|q|=R.$
Let $K=aR$, the conclusion holds as desired.
\end{proof}

\begin{remark}
    If the polynomial in Theorem \ref{a} has gaps and the non vanishing coefficients $a_n,a_{n_1},a_{n_2},\dots$ satisfying
    $$|a_0|a^n\leq \dots\leq |a_{n_2}|a^{n-n_2}\leq |a_{n_1}|a^{n-n_1}\leq |a_n|,$$
    then the conclusion still holds.
\end{remark}

For any $q_1=\sum_{\ell=0}^7 x_\ell e_\ell,\ q_2=\sum_{\ell=0}^7 y_\ell e_\ell\in \mathbb{O},$ we can regard them as real vectors in $\mathbb{R}^8,$ and define the angle $\measuredangle(q_1,q_2)$ between $q_1$ and $q_2$ by the equation
$$ \cos{\measuredangle(q_1,q_2) } =\frac{\sum_{\ell=0}^7x_\ell y_\ell}{|q_1||q_2|}.$$
 There are the following geometric properties.

\begin{theorem}\label{Moduli}
     Let $p(q)=a_0+qa_1+q^2a_2+\dots+q^na_n\not \equiv  0$ be a polynomial of degree $n$ with octonionic variable and octonionic coefficients,  satisfying 
    \begin{equation*}
      \measuredangle (a_k,\beta)\leq \alpha\leq \pi/2,\quad  k=0,1,2,\dots n.
    \end{equation*}
    for some real $\beta$ and
    $$|a_0|\leq |a_1| \leq \dots  \leq |a_{n-1}| \leq|a_n| .$$
    Then all the zeros of $p$ lie in 
    $$|q|\leq \cos{\alpha}+\sin{\alpha}+\frac{2\sin{\alpha}}{|a_n|}\sum_{k=0}^{n-1}|a_k|.$$
\end{theorem}

\begin{proof}
    Let $p(q)=a_0+qa_1+q^2a_2+\dots+q^na_n\not \equiv  0$ be a octonionic polynomial of degree $n$. 
    For each $a_k\in \mathbb{O}$ for $k=0,1,\dots,n,$  we can regard it as a real vector in $\mathbb{R}^8.$
    Given the conditions, we can conclude that
    $$\measuredangle(a_k,a_{k-1})\leq 2\alpha\leq \pi \quad \text{for}\ k=0,1,\dots,n.$$
Then
 \begin{equation*}
        \begin{split}
            |a_k-a_{k-1}|^2&=|a_{k}|^2+|a_{k-1}|^2-2|a_{k}||a_{k-1}|\cos{\measuredangle(a_k,a_{k-1})}\\
            &\leq |a_{k}|^2+|a_{k-1}|^2-2|a_{k}||a_{k-1}|\cos{2\alpha}\\
            &=(|a_{k}|-|a_{k-1}|)^2\cos^2{\alpha}+(|a_{k}|+|a_{k-1}|)^2\sin ^2{\alpha}\\
            &\leq \left ((|a_{k}|-|a_{k-1}|)\cos{\alpha }+(|a_{k}|+|a_{k-1}|)\sin{\alpha }\right )^2.
        \end{split}
\end{equation*}
By the monotonicity of the moduli of the coefficients,
\begin{equation}\label{tri}
    |a_k-a_{k-1}|\leq (|a_{k}|-|a_{k-1}|)\cos{\alpha }+(|a_{k}|+|a_{k-1}|)\sin{\alpha }.
\end{equation}

Let $f(q)=p(q)*(1-q)+q^{n+1}a_n,$  by equation $(\ref{tri}),$ if $|q|=1,$ then
\begin{equation*}
    \begin{split}
        |f(q)|&=|a_0+\sum_{k=1}^nq^k(a_k-a_{k-1})|\\
       &\leq  |a_0|+\sum_{k=1}^n|(a_k-a_{k-1})|\\ 
       &\leq |a_0|+\sum_{k=1}^n(|a_{k}|-|a_{k-1}|)\cos{\alpha }+(|a_{k}|+|a_{k-1}|)\sin{\alpha }\\
       &=|a_0|(1-\cos{\alpha}-\sin{\alpha})+|a_n|(\cos{\alpha}+\sin{\alpha})+2\sin{\alpha }\sum_{k=0}^{n-1}|a_{k}|\\ 
       &\leq |a_n|(\cos{\alpha}+\sin{\alpha})+2\sin{\alpha }\sum_{k=0}^{n-1}|a_{k}|.
    \end{split}
\end{equation*}
Using the same approach as in Theorem \ref{E-K}, if $|q|\geq 1,$ we have
$$|f(q)|\leq \left(|a_n|(\cos{\alpha}+\sin{\alpha})+2\sin{\alpha }\sum_{k=0}^{n-1}|a_{k}|\right)|q|^n.$$

Consequently, for all $|q|>1,$
\begin{equation*}
    \begin{split}
       \left| p(q)*(1-q)\right|&=|f(q)-q^{n+1}a_n|\\
       &\geq |a_n||q|^{n+1}-\left(|a_n|(\cos{\alpha}+\sin{\alpha})+2\sin{\alpha }\sum_{k=0}^{n-1}|a_{k}|\right)|q|^{n}\\
       &=|a_n||q^{n}|\left\{|q|-\left(\cos{\alpha}+\sin{\alpha}+\frac{2\sin{\alpha }}{|a_n|}\sum_{k=0}^{n-1}|a_{k}|\right)\right).
    \end{split}
\end{equation*}
Let $$r:=\cos{\alpha}+\sin{\alpha}+\frac{2\sin{\alpha }}{|a_n|}\sum_{k=0}^{n-1}|a_{k}|,$$ 
then $r\geq 1.$
Thus, if $|q|>r,$ then $ p(q)*(1-q)\neq0,$ implies that $p(q)\neq 0,$
by equation (\ref{zero}). This completes the proof.
\end{proof}

\begin{remark}
    With $\alpha=0$, Theorem \ref{Moduli} reduces to Theorem \ref{E-K}.
\end{remark}

\begin{corollary}
  Let $p(q)=a_0+qa_1+q^2a_2+\dots+q^na_n\not \equiv  0$ be a polynomial of degree $n$ with octonionic variable and octonionic coefficients,  satisfying 
    \begin{equation*}
      \measuredangle (a_k,\beta)\leq \alpha\leq \pi/2,\quad  k=0,1,2,\dots n.
    \end{equation*}
    for some real $\beta$ and
    $$|a_0|\geq |a_1| \geq \dots  \geq |a_{n-1}| \geq|a_n| .$$
    Then all the zeros of $p$ lie in 
    $$|q|\geq \left\{\cos{\alpha}+\sin{\alpha}+\frac{2\sin{\alpha}}{|a_0|}\sum_{k=1}^{n}|a_k|\right \}^{-1}.$$  
\end{corollary}

\begin{proof}
    Let $f(q)=q^np(1/q).$ By the  Artin Theorem in \cite{schafer2017introduction}, we get
    \begin{equation*}
        \begin{split}
           f(q)&=q^n(a_0+\frac{1}{q}a_1+\frac{1}{q^2}a_2+\dots+\frac{1}{q^n}a_n ) \\
           &=a_n+qa_{n-1}+q^2a_{n-2}+\dots+q^na_0. 
        \end{split}
    \end{equation*}
Applying Theorem \ref{Moduli}, we get the result as desired.
\end{proof}

\begin{theorem}
    Let $p(q)=a_0+qa_1+q^2a_2+\dots+q^na_n\not \equiv  0$ be an octonionic polynomial of degree $n$, where $a_k=\sum_{\ell=0}^7a_{k,\ell}e_{ \ell}\  \text{for}\ k=0,1,2,\dots,n.$ If 
\begin{equation}\label{coefficient}
    0 \leq a_{0,0}\leq a_{1,0}\leq \dots \leq a_{n,0}, \quad  a_{n,0}\neq 0.
\end{equation}
    Then all the zeros of $p$ lie in 
    $$|q|\leq 1+\frac{2}{a_{n,0}}\sum_{k=0}^n\sum_{\ell=1}^7|a_{k,\ell}|.$$
\end{theorem}

\begin{proof}
Let $f(q)=p(q)*(1-q)+q^{n+1}a_{n,0},$ then 
\begin{equation*}
    \begin{split}
      f(q)&=a_0+\sum_{k=1}^nq^k(a_k-a_{k-1})-q^{n+1}a_n+q^{n+1}a_{n,0}
      \\&= a_0+\sum_{k=1}^nq^k(a_k-a_{k-1})-q^{n+1}\sum_{\ell=1}^7a_{n,\ell}e_{\ell}.  
    \end{split}
\end{equation*}

If $|q|=1,$ then 
\begin{equation}\label{f}
    \begin{split}
      |f(q)|&\leq  |a_0|+\sum_{k=1}^n|a_k-a_{k-1}|+\sum_{\ell=1}^7|a_{n,\ell}|\\
      &=|a_0|+\sum_{k=1}^n|a_k-a_{k-1}|+\sum_{\ell=1}^7|a_{n,\ell}|
    \end{split}
\end{equation}

Since
$$|a_0|\leq \sum_{\ell=0}^7|a_{0,\ell}|$$
and for $k=1,2,\dots,n,$ according to equation $(\ref{coefficient}),$
\begin{equation*}
    \begin{split}
       |a_k-a_{k-1}|&=\left|\sum_{\ell=0}^7(a_{k,\ell}-a_{k-1,\ell})e_{\ell}\right|\\ &\leq  |a_{k,0}-a_{k-1,0}|+\sum_{\ell=1}^7|a_{k,\ell}|+|a_{k-1,\ell}|\\ &\leq  (a_{k,0}-a_{k-1,0})+\sum_{\ell=1}^7|a_{k,\ell}|+|a_{k-1,\ell}|.
    \end{split}
\end{equation*}

Combining with equation $(\ref{f}),$ we get for $|q|=1,$
\begin{equation*}
    \begin{split}
        |f(q)|&\leq \sum_{\ell=0}^7|a_{0,\ell}|+\sum_{k=1}^n\left(a_{k,0}-a_{k-1,0}+\sum_{\ell=1}^7|a_{k,\ell}|+|a_{k-1,\ell}|\right)+\sum_{\ell=1}^7|a_{n,\ell}| \\&=\sum_{\ell=1}^7|a_{0,\ell}|+a_{n,0}+\sum_{k=1}^n\left (\sum_{\ell=1}^7|a_{k,\ell}|+|a_{k-1,\ell}|\right)+\sum_{\ell=1}^7|a_{n,\ell}| \\ &=a_{n,0}+2\sum_{k=0}^n\sum_{\ell=1}^7|a_{k,\ell}|. 
    \end{split}
\end{equation*}
 Using the same approach as in Theorem \ref{E-K}, if $|q|\geq 1,$ we have
$$|f(q)|\leq \left(a_{n,0}+2\sum_{k=0}^n\sum_{\ell=1}^7|a_{k,\ell}|\right)|q|^n.$$
Consequently, for all $|q|>1,$
\begin{equation*}
    \begin{split}
       \left| p(q)*(1-q)\right|&=|f(q)-q^{n+1}a_{n,0}|\\ &\geq |a_{n,0}||q|^{n+1}-|f(q)|\\
       &\geq a_{n,0}|q|^{n+1}-\left(a_{n,0}+2\sum_{k=0}^n\sum_{\ell=1}^7|a_{k,\ell}|\right)|q|^{n}\\
       &=a_{n,0}|q^{n}|\left\{|q|-\left(1+\frac{2}{a_{n,0}}\sum_{k=0}^n\sum_{\ell=1}^7|a_{k,\ell}|\right)\right).
    \end{split}
\end{equation*}
Let $$r:=1+\frac{2}{a_{n,0}}\sum_{k=0}^n\sum_{\ell=1}^7|a_{k,\ell}|,$$ 
then $r\geq 1.$
Thus, if $|q|>r,$ then $ p(q)*(1-q)\neq0,$ implies that $p(q)\neq 0.$ This completes the proof.
\end{proof}

\subsection*{Funding}{This work was supported by the National Natural Science Foundation of China(12401104), the Fundamental Research Funds for the Central Universities (WK0010000091) and Xiaomi Young Talents Program.}

\bibliographystyle{amsplain}
\bibliography{mybibfile}
\printindex

\end{document}